\documentclass[11pt]{article}
\usepackage{cite}
\usepackage{amsmath,amsthm,amssymb,amsfonts,amscd}
\usepackage{enumerate}
\usepackage{graphicx}
\usepackage{graphics}
\usepackage{subfigure}
\usepackage{mathrsfs}
\usepackage{a4,latexsym,epsfig}
\usepackage{color,colordvi}
\usepackage{tikz} 
\usepackage{fullpage}
\usepackage[ruled,vlined]{algorithm2e}
\usepackage{soul}
\setstcolor{red}
\usepackage{cancel}
\usepackage[normalem]{ulem}

\newtheorem{theorem}{Theorem}[section]

\newtheorem{lemma}[theorem]{Lemma}
\newtheorem{conjecture}{Conjecture}

\title{Cyclic Base Ordering of Graphs\thanks{This is a summer research project (2021) of three high school students Jessica Li, Eric Yang and William Zhang under the supervision of Dr. Xiaofeng Gu of University of West Georgia.}
}
\author{Jessica Li$^1$, Eric Yang$^2$, William Zhang$^3$
\thanks{The authors are listed in alphabetical order by last name.}
\\
\\
\small $^1$Westwood High School\\\small Round Rock, Texas 78750\\\small Email: jessicaalbertaone.li@gmail.com
\\
\small $^2$Allen High School\\\small Allen, Texas 75002\\\small Email: eyangch@gmail.com
\\
\small $^3$Harker High School\\ \small San Jose, California 95129\\ \small Email: wzyang24@gmail.com
}

\begin{document}
\date{}
\maketitle

\begin{abstract}
\normalsize
A cyclic base ordering of a connected graph $G$, is a cyclic ordering of $E(G)$ such that every cyclically consecutive $|V(G)|-1$ edges form a spanning tree. In this project, we study cyclic base ordering of various families of graphs, including square of cycles, wheel graphs, generalized wheel graphs and broken wheel graphs, fan and broken fan graphs, prism graphs, and maximal 2-degenerate graphs. We also provide a polynomial time algorithm to verify any giving edge ordering is a cyclic base ordering.
\end{abstract}


\newpage
\tableofcontents
\newpage
\section{Introduction}
Let $G$ be a connected graph on $n$ vertices with vertex set $V(G)$ and edge set $E(G)$. A {\bf cyclic base ordering} or shortly {\bf CBO} of $G$, is a cyclic ordering of $E(G)$ such that every cyclically consecutive $n-1$ edges form a spanning tree. Equivalently, a cyclic base ordering is a bijection $\mathcal{O}: E(G) \longrightarrow \{1,2,\ldots,|E(G)|\}$ such that $\{\mathcal{O}^{-1}(k): k= i, i+1,\ldots, i+n-2\}$ forms a spanning tree of $G$, where the labelling $k$ is equivalent modulo $|E(G)|$. Here is a simple fact. If a graph has a CBO, then the CBO can start with any edge, since it is cyclically ordered.

By definition, it is easy to see that any path, cycle, and tree has a CBO.
The well-known Petersen graph has a CBO, as shown in Figure~\ref{fig:petersen}.
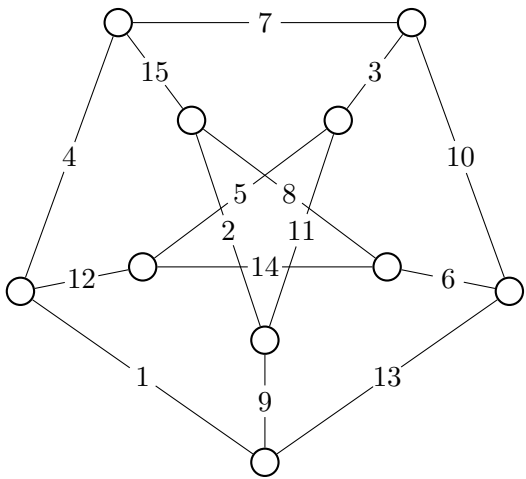
\begin{figure}[htb]
\begin{center}
\begin{tikzpicture}[every node/.style={circle,thick,draw}, scale=0.65] 
\node (1) at (5, 0) {};
\node (2) at (0, 3.5) {};
\node (3) at (2, 9) {};
\node (4) at (8, 9) {};
\node (5) at (10, 3.5) {};
\node (6) at (5, 2.5) {};
\node (7) at (2.5, 4) {};
\node (8) at (3.5, 7) {};
\node (9) at (6.5, 7) {};
\node (10) at (7.5, 4) {};
\begin{scope}[>={},every node/.style={fill=white,circle,inner sep=0pt,minimum size=12pt}]
\path [] (1) edge node {1} (2);
\path [] (8) edge node {2} (6);
\path [] (4) edge node {3} (9);
\path [] (2) edge node {4} (3);
\path [] (7) edge node {5} (9);
\path [] (5) edge node {6} (10);
\path [] (3) edge node {7} (4);
\path [] (8) edge node {8} (10);
\path [] (1) edge node {9} (6);
\path [] (4) edge node {10} (5);
\path [] (6) edge node {11} (9);
\path [] (2) edge node {12} (7);
\path [] (1) edge node {13} (5);
\path [] (7) edge node {14} (10);
\path [] (3) edge node {15} (8);
\end{scope}
\end{tikzpicture}
\end{center}
\caption{CBO of the Petersen graph}
\label{fig:petersen}
\end{figure}

The {\bf density} of $G$, denoted by $d(G)$, is defined to be $$d(G)=\frac{|E(G)|}{|V(G)|-1}.$$
A graph $G$ is {\bf uniformly dense} if $d(H)\leq d(G)$ for every subgraph $H$ of $G$.
\begin{conjecture}[Kajitani, Ueno and Miyano~\cite{KaUM88}]
A connected graph $G$ has a cyclic base ordering if and only if $G$ is uniformly dense.
\end{conjecture}
The necessity was confirmed in \cite{KaUM88}, however, the sufficiency is still unsolved. To support this conjecture, the following results have been proved in \cite{KaUM88} and \cite{GuHL14}.
\begin{theorem}[\hspace{1sp}\cite{KaUM88,GuHL14}]
\label{prev:CBO}
The following graphs have cyclic base ordering.
\begin{itemize}
    \item Any complete graph $K_n$.
    \item Any complete bipartite graph.
    \item Any $k$-tree for $k\in\{2,3\}$ (A graph $G$ is a {\bf $k$-tree} if $G=K_{k+1}$ or $G$ has a vertex $v$ such that $G-v$ is a $k$-tree and such that $v$ is adjacent to all vertices in a clique of order $k$).
    \item Any (multi)graph consisting of exactly 2 edge-disjoint spanning trees.
\end{itemize}
\end{theorem}

We prove that the conjecture is true for several families of uniformly dense graphs in the next several sections, including square of cycles, wheel graphs, generalized wheel graphs and broken wheel graphs, fan and broken fan graphs, prism graphs, as well as maximal 2-degenerate graphs. We also provide a polynomial time algorithm to verify any giving edge ordering is a cyclic base ordering.

Given an edge ordering of a graph $G$, every cyclically consecutive $|V(G)|-1$ edges is called a {\bf progression}. To verify an edge ordering is a CBO, it suffices to show that any progression of this edge ordering induces a spanning tree.

\section{The square of cycles}
The {\bf square of a cycle $C_n$}, denoted by $C_n^2$, is a graph obtained by joining every pair of vertices of distance two in $C_n$.

\begin{theorem}
Every graph $C_n^2$ has a cyclic base ordering.
\end{theorem}
\begin{proof}
Suppose that the vertices of $C_n^2$ are $v_1,v_2,\ldots, v_n$ in the order on the cycle $C_n$. For each $i=1,2,\ldots, n$, let $a_i$ be the edge joining $v_i$ and $v_{i+2}$, and $b_i$ be the edge joining $v_i$ and $v_{i+1}$, where the subscripts are equivalent modulo $n$.

For $n=2k+1$, we define $\mathcal{O}_{2k+1} = (a_1, a_2, \ldots, a_{2k}, a_{2k+1}, b_1, b_3, b_2, b_5, b_4, \ldots, b_{2k+1}, b_{2k})$. An example is shown in Figure~\ref{fig:oddsq} when $n=7$. We can show $\mathcal{O}_{2k+1}$ is a cyclic base ordering of $C_{2k+1}^2$.
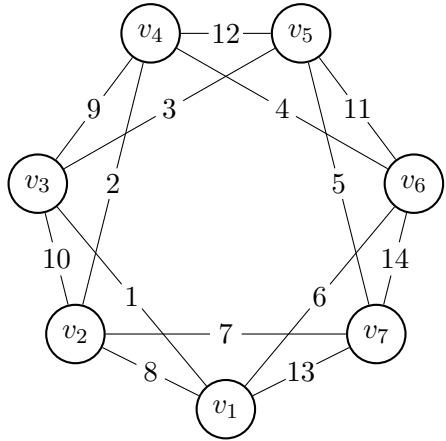
\begin{figure}[htb]
\centering
\begin{tikzpicture}[every node/.style={circle,thick,draw}] 
\node (1) at (2.5, 0) {$v_1$};
\node (2) at (0.5, 1) {$v_2$};
\node (3) at (0, 3) {$v_3$};
\node (4) at (1.5, 5) {$v_4$};
\node (5) at (3.5, 5) {$v_5$};
\node (6) at (5, 3) {$v_6$};
\node (7) at (4.5, 1) {$v_7$};
\begin{scope}[>={},every node/.style={fill=white,circle,inner sep=0pt,minimum size=12pt}]
\path [] (1) edge node {1} (3);
\path [] (2) edge node {2} (4);
\path [] (3) edge node {3} (5);
\path [] (4) edge node {4} (6);
\path [] (5) edge node {5} (7);
\path [] (6) edge node {6} (1);
\path [] (7) edge node {7} (2);
\path [] (1) edge node {8} (2);
\path [] (3) edge node {9} (4);
\path [] (2) edge node {10} (3);
\path [] (5) edge node {11} (6);
\path [] (4) edge node {12} (5);
\path [] (7) edge node {13} (1);
\path [] (6) edge node {14} (7);
\end{scope}
\end{tikzpicture}
\caption{CBO of $C_7^2$}
\label{fig:oddsq}
\end{figure}

For an odd $t\ge 5$, any progression starting with $a_t$ is in the form of 
$$a_t,\ldots, a_{2k}, a_{2k+1}, b_1,b_3,\cdots,b_{t-4}, b_{t-5}, b_{t-2}, b_{t-3}.$$
All edges $a_i$ in this progression form two paths $v_t v_{t+2}\ldots v_{2k+1} v_2$ and $v_{t+1} v_{t+3}\ldots v_{2k} v_1$, while all edges $b_i$ form a path $v_1 v_2\ldots v_{t-1}$. Clearly, $v_1$ and $v_2$ are the only common vertices of the three paths and all vertices are connected through the three paths. Thus it is a spanning tree.

For an odd $t\ge 5$, any progression starting with $b_t$ is in the form of 
$$b_t,b_{t-1}, b_{t+2}, b_{t+1},\ldots,b_{2k+1},b_{2k},a_1,\ldots,a_{t-3}.$$
All edges $a_i$ in this progression form two paths $v_1 v_3\ldots v_{t-2}$ and $v_2 v_4\ldots v_{t-1}$, while all edges $b_i$ form a path $v_{t-1} v_t\ldots v_{2k+1} v_1$. Clearly, the three paths form a longer spanning path, and thus is a spanning tree.

For an even $t\ge 4$, any progression starting with $a_t$ is in the form of 
$$a_t,\ldots, a_{2k}, a_{2k+1}, b_1,b_3,\cdots,b_{t-3}, b_{t-4}, b_{t-1}.$$
All edges $a_i$ in this progression form two paths $v_t v_{t+2}\ldots v_{2k} v_1$ and $v_{t+1} v_{t+3}\ldots v_{2k+1} v_2$, while all edges $b_i$ form two paths $v_1 v_2\ldots v_{t-2}$ and $v_{t-1} v_t$. The two paths formed by edges $b_i$ are connected by the path $v_t v_{t+2} \ldots v_{2k} v_1$, and the final path $v_{t+1} v_{t+3} \ldots v_{2k+1} v_2$ only shares one vertex with any of the other 3 paths, so the progression induces a spanning tree.

For an even $t\ge 4$, any progression starting with $b_t$ is in the form of 
$$b_t,b_{t+3}, b_{t+2}, b_{t+5},\ldots, b_{2k+1},b_{2k},a_1,\ldots,a_{t-1}.$$
All edges $a_i$ in this progression form two paths $v_1 v_3\ldots v_{t-1} v_{t+1}$ and $v_2 v_4\ldots v_{t-2} v_t$, while all edges $b_i$ form two paths $v_t v_{t+1}$ and $v_{t+2} v_{t+3}\ldots v_{2k+1} v_1$. Since the four paths form a longer spanning path, the progression induces a spanning tree.

Thus $\mathcal{O}_{2k+1}$ is a cyclic base ordering of $C_{2k+1}^2$.

\medskip
For $n=2k$, define $\mathcal{O}_{2k} = (a_1, b_{k+1}, a_2, b_{k+2}, \ldots, a_k, b_{2k}, a_{k+1}, b_1, \ldots, a_{2k}, b_{k})$. An example is shown in Figure~\ref{fig:evensq} when $n=6$. We can show $O_{2k}$ is a cyclic base ordering of $C_{2k}^2$.
\begin{figure}[htb]
\centering{
\begin{tikzpicture}[every node/.style={circle,thick,draw}] 
\node (1) at (3.75, 0.5) {$v_1$};
\node (2) at (1.25, 0.5) {$v_2$};
\node (3) at (0, 2.5) {$v_3$};
\node (4) at (1.25, 4.5) {$v_4$};
\node (5) at (3.75, 4.5) {$v_5$};
\node (6) at (5, 2.5) {$v_6$};
\begin{scope}[>={},every node/.style={fill=white,circle,inner sep=0pt,minimum size=12pt}]
\path [] (1) edge node {1} (3);
\path [] (4) edge node {2} (5);
\path [] (2) edge node {3} (4);
\path [] (5) edge node {4} (6);
\path [] (3) edge node {5} (5);
\path [] (1) edge node {6} (6);
\path [] (4) edge node {7} (6);
\path [] (1) edge node {8} (2);
\path [] (1) edge node {9} (5);
\path [] (2) edge node {10} (3);
\path [] (2) edge node {11} (6);
\path [] (3) edge node {12} (4);
\end{scope}
\end{tikzpicture}
}
\caption{CBO of $C_6^2$}
\label{fig:evensq}
\end{figure}
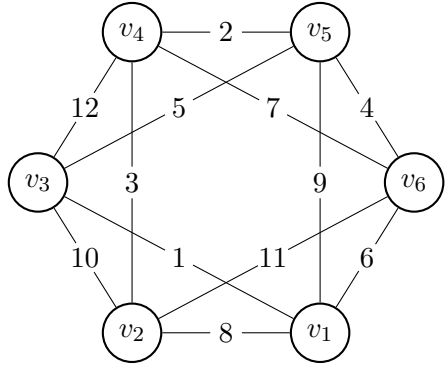

We can use a similar idea to show that $\mathcal{O}_{2k}$ is a CBO. A progression starting with $a_t$ is in the form of 
$$a_t, b_{k+t}, a_{t+1}, b_{k+t+1},\ldots, a_{k+t-2}, b_{t-2}, a_{k+t-1},$$ 
where the subscripts are modulo $n$. All edges $a_i$'s in this progression form two paths $v_{k+t} v_{k+t-2}\ldots$ and $v_{k+t+1} v_{k+t-1}\ldots$, while all edges $b_i$'s form the path $v_{k+t} v_{k+t+1}\ldots$. The final path shares one vertex with each of the first two paths, so the progression forms a spanning tree.

A progression starting with $b_t$ is in the form of
$$b_t, a_{k+t+1}, b_{t+1},a_{k+t+2}, \ldots, b_{k+t-1}, a_{t-1}, b_{k+t-1},$$
where the subscripts are modulo $n$. All edges $a_i$'s in this progression form two paths $v_{t-2} v_{t-4} \ldots$ and $v_{t-1} v_{t-3}\ldots$, while all edges $b_i$'s form the path $v_{t} v_{t+1}\ldots v_{k+t}$. The final path shares one vertex with each of the first two paths, so the progression forms a spanning tree.

Therefore $O_{2k}$ is a cyclic base ordering of $C_{2k}^2$, completing the proof.
\end{proof}

\section{Wheel and generalized wheel graphs}
A {\bf wheel graph} on $n\ge 4$ vertices, denoted by $W_n$, is a graph obtained by joining a single universal vertex to all vertices of a cycle $C_{n-1}$. This cycle is call the {\bf rim} of $W_n$, and each edge joining the universal vertex and a vertex on the rim is called a {\bf spoke}.

\subsection{CBO of wheel graphs}
Clearly, any wheel graph $W_n$ consists of exactly two edge-disjoint spanning trees, and thus $W_n$ has a CBO by Theorem~\ref{prev:CBO}. Here we give an explicit construction of a CBO for $W_n$.

\begin{theorem}
Every wheel graph $W_n$ has a cyclic base ordering.
\end{theorem}
\begin{proof}
Suppose that the vertices of $W_n$ are $v_1,v_2,\ldots, v_n$ with $v_n$ being the universal vertex. Suppose that vertices $v_1,v_2,\ldots, v_{n-1}$ are in the order on the rim of $W_n$. For each $i=1,2,\ldots, n-1$, let $a_i$ be the edge $v_iv_n$. For each $i=1,2,\ldots, n-2$, let $b_i$ be the edge $v_iv_{i+1}$ and $b_{n-1} = v_{n-1}v_1$.

Case 1 : $n = 2k$. 
We can construct an ordering of $E(G)$ by inserting $ b_i$ between $ a_{i-(k-1)}$ and $a_{i-(k-2)}$, and get $$\mathcal{O}_{2k} = (a_1, b_{k}, a_2, b_{k+1}, \ldots, a_k, b_{2k-1}, a_{k+1}, b_{2k}, \ldots, a_{2k-1}, b_{k-1}).$$

Figure~\ref{fig:w6} shows an ordering for $W_6$.
\begin{figure}[htb]
\centering{
\begin{tikzpicture}[every node/.style={circle,thick,draw}] 
\node (1) at (2.5, 0) {$v_1$};
\node (2) at (0, 1.75) {$v_2$};
\node (3) at (1, 4.5) {$v_3$};
\node (4) at (4, 4.5) {$v_4$};
\node (5) at (5, 1.75) {$v_5$};
\node (6) at (2.5, 2.25) {$v_6$};
\begin{scope}[>={},every node/.style={fill=white,circle,inner sep=0pt,minimum size=12pt}]
\path [] (1) edge node {1} (6);
\path [] (3) edge node {2} (4);
\path [] (2) edge node {3} (6);
\path [] (4) edge node {4} (5);
\path [] (3) edge node {5} (6);
\path [] (5) edge node {6} (1);
\path [] (4) edge node {7} (6);
\path [] (1) edge node {8} (2);
\path [] (5) edge node {9} (6);
\path [] (2) edge node {10} (3);
\end{scope}
\end{tikzpicture} 
}
\caption{CBO of $W_6$}
\label{fig:w6}
\end{figure}
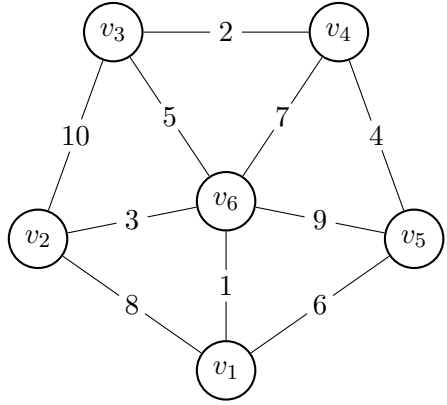

Since this ordering is symmetric, all progressions starting at $a_t$ have the same structure, for $t=1,2,\ldots, 2k-1$. Similarly, all progressions starting at $b_t$ also have the same structure. Thus it suffices to check the first two progressions in $\mathcal{O}_{2k}$, i.e., the progression $a_1, b_{k}, a_2, b_{k+1}, \ldots, a_k$ and the progression $b_{k}, a_2, b_{k+1}, \ldots, a_k, b_{2k-1}$.

In the progression $a_1, b_{k}, a_2, b_{k+1}, \ldots, a_k$, all edges $a_i$ induce a star with vertices $v_1, v_2,\cdots, v_k, v_{2k}$ where $v_{2k}$ is the center, and all edges $b_i$ induce a path $v_k v_{k+1}\cdots v_{2k-1}$.
Clearly, the star and the path only share one common vertex $v_k$, and thus the progression induces a spanning tree.

In the progression $b_{k}, a_2, b_{k+1}, \ldots, a_k, b_{2k-1}$, all edges $a_i$ induce a star with vertices $v_2,\cdots, v_k, v_{2k}$ where $v_{2k}$ is the center, and all edges $b_i$ induce a path $v_k v_{k+1}\cdots v_{2k-1}  v_1$. Clearly, the star and the path only share one common vertex $v_k$, and thus the progression induces a spanning tree.

Thus $\mathcal{O}_{2k}$ is a cyclic base ordering of $G$.

\bigskip
Case 2: $n = 2k+1$. Similarly, we can construct an ordering of $E(G)$ by inserting $ b_i$ between $ a_{i-k}$ and $a_{i-(k-1)}$, and get
$$\mathcal{O}_{2k+1} = (a_1, b_{k+1}, a_2, b_{k+2}, \ldots, a_k, b_{2k}, a_{k+1}, b_{1}, \ldots, a_{2k-1}, b_{k-1},a_{2k},b_{k}).$$
Figure~\ref{fig:w7} shows an ordering for $W_7$.
\begin{figure}[htb]
\centering{
\begin{tikzpicture}[every node/.style={circle,thick,draw}] 
\node (1) at (3.75, 0.5) {$v_1$};
\node (2) at (1.25, 0.5) {$v_2$};
\node (3) at (0, 2.5) {$v_3$};
\node (4) at (1.25, 4.5) {$v_4$};
\node (5) at (3.75, 4.5) {$v_5$};
\node (6) at (5, 2.5) {$v_6$};
\node (7) at (2.5, 2.5) {$v_7$};
\begin{scope}[>={},every node/.style={fill=white,circle,inner sep=0pt,minimum size=12pt}]
\path [] (1) edge node {1} (7);
\path [] (4) edge node {2} (5);
\path [] (2) edge node {3} (7);
\path [] (5) edge node {4} (6);
\path [] (3) edge node {5} (7);
\path [] (6) edge node {6} (1);
\path [] (4) edge node {7} (7);
\path [] (1) edge node {8} (2);
\path [] (5) edge node {9} (7);
\path [] (2) edge node {10} (3);
\path [] (6) edge node {11} (7);
\path [] (3) edge node {12} (4);
\end{scope}
\end{tikzpicture}
}
\caption{CBO of $W_7$}
\label{fig:w7}
\end{figure}
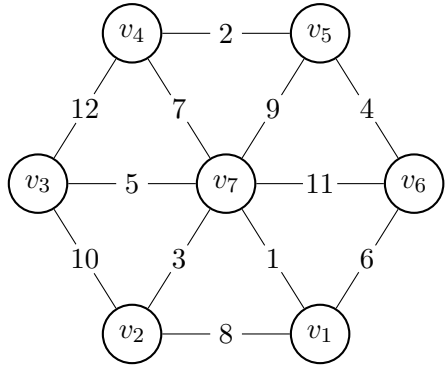

Since $\mathcal{O}_{2k+1}$ is symmetric, all progressions starting at $a_t$ or $b_t$ have the same structure. Thus it suffices to verify the first two progressions. In the progression $a_1, b_{k+1}, a_2, b_{k+2}, \ldots, a_k, b_{2k}$, all edges $a_i$'s induce a star with vertices $v_1, v_2,\cdots, v_k, v_{2k+1}$ where $v_{2k+1}$ is the center, and all edges $b_i$'s induce a path $v_{k+1} v_{k+2}\cdots v_{2k} v_1$. Clearly, the star and the path only share one common vertex $v_1$, and thus the progression induces a spanning tree.

In the progression $b_{k+1},a_2,b_{k+2},a_3,\ldots,b_{2k}, a_{k+1}$, all edges $a_i$'s induce a star with vertices $v_2, v_3\cdots, v_{k+1}, v_{2k+1}$ where $v_{2k+1}$ is the center, and all edges $b_i$'s induce a path $v_{k+1} v_{k+2}\cdots v_{2k} v_1$. Clearly, the star and the path only share one common vertex $v_{k+1}$, and thus the progression induces a spanning tree.

Thus $\mathcal{O}_{2k+1}$ is a cyclic base ordering of $G$, completing the proof.
\end{proof}

 \subsection{CBO of double wheel graphs}
A {\bf double wheel graph} on $n\ge 5$ vertices, denoted by $W^2_n$, is a graph obtained by joining two single universal vertices to all vertices of a cycle $C_{n-2}$. Any double wheel graph on $n$ vertices has $3n-6$ edges.

\begin{theorem}
Every double wheel graph has a cyclic base ordering.
\end{theorem}
\begin{proof}
Let $G$ be a double wheel graph on $n$ vertices. Suppose that the vertices of $G$ are $v_0$, $v_1$,$\ldots$,$v_{n-1}$ with $v_0$ and $v_{n-1}$ being the two universal vertices, and $v_1, v_2,\ldots, v_{n-2}$ are in order on the rim of $G$. For each $i = 1, 2,\ldots, n-2$, let $a_i$ be the edge $v_iv_0$ and $c_i$ be the edge $v_iv_{n-1}$. For each $i = 1, 2,\ldots, n-3$, let $b_i$ be the edge $v_iv_{i+1}$ and $b_{n-2} = v_{n-2}v_1$.

We divide into three cases: $n=3k$, $n=3k+1$ and $n=3k+2$. 

{\bf Case 1: $n=3k$}. We define an edge ordering $\mathcal{O}$ of $G$ by 
$$\mathcal{O} = (a_1, b_{k}, c_{2k}, a_2, b_{k+1}, c_{2k+1}, \ldots, a_{n-2}, b_{k-1}, c_{2k-1}).$$
In other words, it is an ordered list of $a_t, b_{t+k-1}, c_{t+2k-1}$ for $t=1,2,\ldots, n-2$. Notice that any subscript is at most $n-2$. If some subscript $i$ is larger than $n-2$, we use $i\pmod{n-2}$. An example for $n=6$ is shown in Figure~\ref{fig:dw6}.

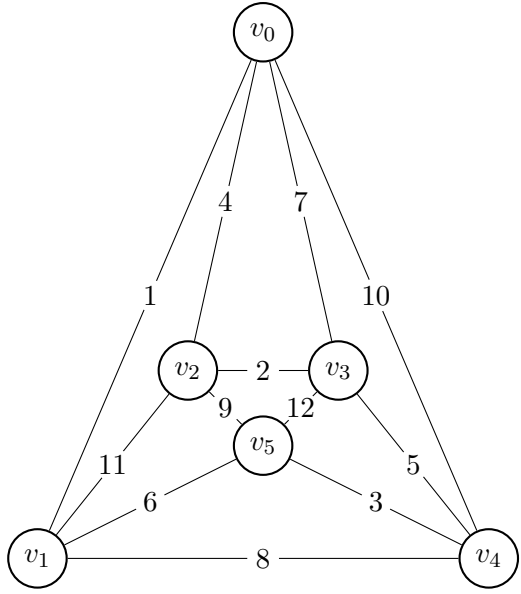
\begin{figure}[htb]
\centering{
\begin{tikzpicture}[every node/.style={circle,thick,draw}] 
\node (1) at (3, 7) {$v_0$};
\node (2) at (0, 0) {$v_1$};
\node (3) at (2, 2.5) {$v_2$};
\node (4) at (4, 2.5) {$v_3$};
\node (5) at (6, 0) {$v_4$};
\node (6) at (3, 1.5) {$v_5$};
\begin{scope}[>={},every node/.style={fill=white,circle,inner sep=0pt,minimum size=12pt}]
\path [] (1) edge node {1} (2);
\path [] (1) edge node {4} (3);
\path [] (1) edge node {7} (4);
\path [] (1) edge node {10} (5);
\path [] (3) edge node {2} (4);
\path [] (4) edge node {5} (5);
\path [] (5) edge node {8} (2);
\path [] (2) edge node {11} (3);
\path [] (5) edge node {3} (6);
\path [] (2) edge node {6} (6);
\path [] (3) edge node {9} (6);
\path [] (4) edge node {12} (6);
\end{scope}
\end{tikzpicture}
}
\caption{CBO of $W^2_6$}
\label{fig:dw6}
\end{figure}

By symmetry, all progressions starting at any $a_t$ have the same structure, and similarly for progressions starting at $b_t$ and $c_t$. Thus it suffices to verify the first three progressions in $\mathcal{O}$. 

Notice that every progression contains $n-1 = 3k-1$ edges. The first progression is
$$(a_1, b_{k}, c_{2k}, a_2, b_{k+1}, c_{2k+1}, \ldots, a_{k-1}, b_{2k-2}, c_{3k-2}, a_k, b_{2k-1}).$$
The edges $a_1, a_2,\ldots, a_k$ induce a star by joining $v_1,v_2,\ldots,v_k$ to $v_0$. The edges $b_k, b_{k+1},\ldots, b_{2k-1}$ induce a path $v_kv_{k+1}\cdots v_{2k}$, while the edges $c_{2k}, c_{2k+1},\ldots,c_{3k-2}$ (notice that $3k-2=n-2$) induce a star by joining $v_{2k},v_{2k+1},\ldots,v_{n-2}$ to the vertex $v_{n-1}$. Clearly all vertices are connected to form a spanning tree.

The second progression is
$$(b_{k}, c_{2k}, a_2, b_{k+1}, c_{2k+1}, \ldots, a_{k-1}, b_{2k-2}, c_{3k-2}, a_k, b_{2k-1}, c_{3k-1}).$$
The edges $a_2,\ldots, a_k$ induce a star by joining $v_2,\ldots,v_k$ to $v_0$, and the edges $b_k, b_{k+1},\ldots, b_{2k-1}$ induce a path $v_kv_{k+1}\cdots v_{2k}$. The edges $c_{2k}, c_{2k+1},\ldots,c_{3k-2}, c_{3k-1}$ (notice that $3k-2=n-2$ and $3k-1=n-1 > n-2$; thus $c_{3k-2}$ is $c_{n-2}$ and $c_{3k-1}$ is actually $c_1$) induce a star by joining $v_{2k},v_{2k+1},\ldots,v_{n-2}$ and $v_1$ to the vertex $v_{n-1}$. Clearly all vertices are connected to form a spanning tree.

The third progression is
$$(c_{2k}, a_2, b_{k+1}, c_{2k+1}, \ldots, a_{k-1}, b_{2k-2}, c_{3k-2}, a_k, b_{2k-1}, c_{3k-1}, a_{k+1}).$$
The edges $a_2,\ldots, a_k, a_{k+1}$ induce a star by joining $v_2,\ldots,v_k, v_{k+1}$ to the vertex $v_0$, and the edges $b_{k+1},\ldots, b_{2k-1}$ induce a path $v_{k+1}\cdots v_{2k}$. The edges $c_{2k}, c_{2k+1},\ldots,c_{3k-2}, c_{3k-1}$ (notice that $3k-2=n-2$ and $3k-1=n-1 > n-2$; thus $c_{3k-2}$ is $c_{n-2}$ and $c_{3k-1}$ is actually $c_1$) induce a star by joining $v_{2k},v_{2k+1},\ldots,v_{n-2}$ and $v_1$ to the vertex $v_{n-1}$. Clearly all vertices are connected to form a spanning tree.

{\bf Case 2: $n=3k+1$}. We define an edge ordering $\mathcal{O}$ of $G$ by
$$\mathcal{O} = (a_1, c_{2k+1}, b_{k+1}, a_2, c_{2k+2}, b_{k+2},\ldots, a_{n-2}, c_{2k}, b_{k}).$$
In other words, it is an ordered list of $a_t, c_{t+2k}, b_{t+k}$ for $t=1,2,\ldots, n-2$. Notice that any subscript is at most $n-2$. If some subscript $i$ is larger than $n-2$, we use $i\pmod{n-2}$. An example for $n=7$ is shown in Figure~\ref{fig:dw7}.

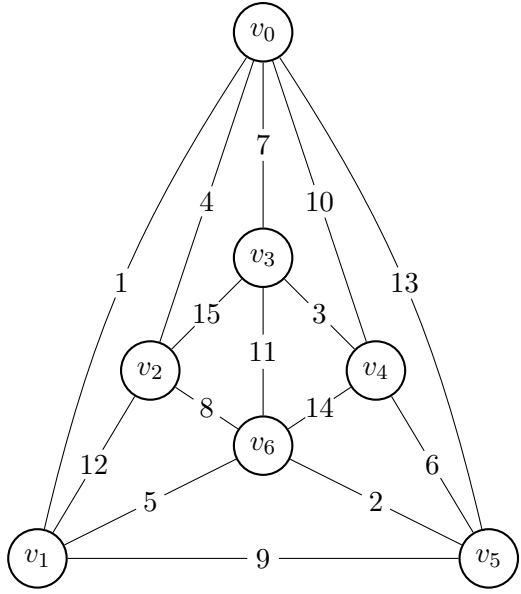
\begin{figure}[htb]
\centering{
\begin{tikzpicture}[every node/.style={circle,thick,draw}] 
\node (1) at (3, 7) {$v_0$};
\node (2) at (0, 0) {$v_1$};
\node (3) at (1.5, 2.5) {$v_2$};
\node (4) at (3, 4) {$v_3$};
\node (5) at (4.5, 2.5) {$v_4$};
\node (6) at (6, 0) {$v_5$};
\node (7) at (3, 1.5) {$v_6$};
\begin{scope}[>={},every node/.style={fill=white,circle,inner sep=0pt,minimum size=12pt}]
\path [] (1) edge[bend right=10] node {1} (2);
\path [] (3) edge node {15} (4);
\path [] (6) edge node {2} (7);
\path [] (1) edge node {4} (3);
\path [] (4) edge node {3} (5);
\path [] (2) edge node {5} (7);
\path [] (1) edge node {7} (4);
\path [] (5) edge node {6} (6);
\path [] (3) edge node {8} (7);
\path [] (1) edge node {10} (5);
\path [] (6) edge node {9} (2);
\path [] (7) edge node {11} (4);
\path [] (6) edge[bend right=10] node {13} (1);
\path [] (3) edge node {12} (2);
\path [] (5) edge node {14} (7);
\end{scope}
\end{tikzpicture}
}
\caption{CBO of $W^2_7$}
\label{fig:dw7}
\end{figure}

Every progression contains $n-1 = 3k$ edges. The first three progressions are
$$(a_1, c_{2k+1}, b_{k+1}, a_2, c_{2k+2}, b_{k+2},\ldots, a_k, c_{3k}, b_{2k}),$$
$$(c_{2k+1}, b_{k+1}, a_2, c_{2k+2}, b_{k+2},\ldots, a_k, c_{3k}, b_{2k}, a_{k+1}),$$
$$(b_{k+1}, a_2, c_{2k+2}, b_{k+2},\ldots, a_k, c_{3k}, b_{2k}, a_{k+1}, c_{3k+1}).$$
We can show that each progression induces a spanning tree. It is quite similar as above, and thus we omit the proof.

{\bf Case 3: $n=3k+2$}. We define an edge ordering $\mathcal{O}$ of $G$ in the same way as Case 2, that is
$$\mathcal{O} = (a_1, c_{2k+1}, b_{k+1}, a_2, c_{2k+2}, b_{k+2},\ldots, a_{n-2}, c_{2k}, b_{k}).$$
In other words, it is an ordered list of $a_t, c_{t+2k}, b_{t+k}$ for $t=1,2,\ldots, n-2$. Notice that any subscript is at most $n-2$. If some subscript $i$ is larger than $n-2$, we use $i\pmod{n-2}$. An example for $n=8$ is shown in Figure~\ref{fig:dw8}.

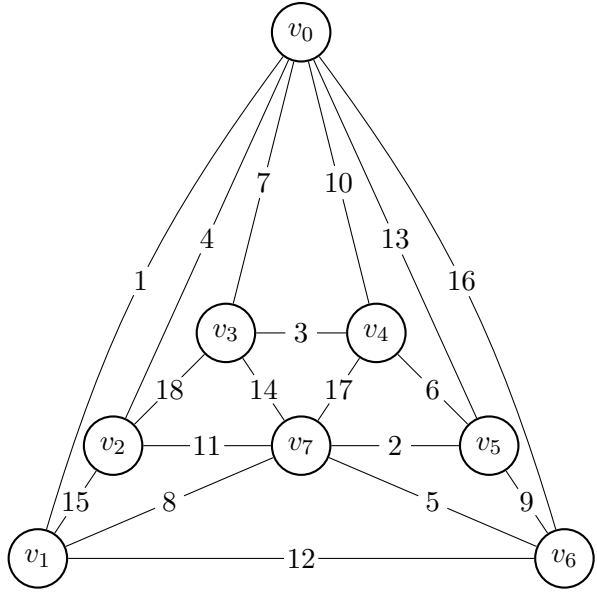
\begin{figure}[htb]
\centering{
\begin{tikzpicture}[every node/.style={circle,thick,draw}] 
\node (1) at (3.5, 7) {$v_0$};
\node (2) at (0, 0) {$v_1$};
\node (3) at (1, 1.5) {$v_2$};
\node (4) at (2.5, 3) {$v_3$};
\node (5) at (4.5, 3) {$v_4$};
\node (6) at (6,1.5) {$v_5$};
\node (7) at (7, 0) {$v_6$};
\node (8) at (3.5,1.5) {$v_7$};
\begin{scope}[>={},every node/.style={fill=white,circle,inner sep=0pt,minimum size=12pt}]
\path [] (1) edge[bend right=10] node {1} (2);
\path [] (1) edge node {4} (3);
\path [] (1) edge node {7} (4);
\path [] (1) edge node {10} (5);
\path [] (1) edge node {13} (6);
\path [] (1) edge[bend left=10] node {16} (7);
\path [] (2) edge node {15} (3);
\path [] (3) edge node {18} (4);
\path [] (4) edge node {3} (5);
\path [] (5) edge node {6} (6);
\path [] (6) edge node {9} (7);
\path [] (7) edge node {12} (2);
\path [] (2) edge node {8} (8);
\path [] (3) edge node {11} (8);
\path [] (4) edge node {14} (8);
\path [] (5) edge node {17} (8);
\path [] (6) edge node {2} (8);
\path [] (7) edge node {5} (8);
\end{scope}
\end{tikzpicture}
}
\caption{CBO of $W^2_8$}
\label{fig:dw8}
\end{figure}

Every progression contains $n-1 = 3k+1$ edges. The first three progressions are
$$(a_1, c_{2k+1}, b_{k+1}, a_2, c_{2k+2}, b_{k+2},\ldots, a_k, c_{3k}, b_{2k}, a_{k+1}),$$
$$(c_{2k+1}, b_{k+1}, a_2, c_{2k+2}, b_{k+2},\ldots, a_k, c_{3k}, b_{2k}, a_{k+1}, c_{3k+1}),$$
$$(b_{k+1}, a_2, c_{2k+2}, b_{k+2},\ldots, a_k, c_{3k}, b_{2k}, a_{k+1}, c_{3k+1}, b_{2k+1}).$$
Each progression induces a spanning tree. The proof is again similar to the above, and thus is omitted.

Therefore $\mathcal{O}$ is a cyclic base ordering of $G$.

\end{proof}

\section{Fan and broken fan graphs}
A {\bf fan graph} on $n$ vertices, denoted by $F_n$, is a graph obtained by joining a single universal vertex to all vertices of a path $P_{n-1}$. Each edge incident to the universal vertex is called a {\bf spoke}. Because any fan graph $F_n$ is a 2-tree, $F_n$ has a cyclic base ordering from Theorem \ref{prev:CBO}.

A {\bf broken fan graph} is a graph obtained from a fan graph by removing some inner spokes. A general broken fan graph may not be uniformly dense, and thus may not have a CBO. We can define the following ``uniformly'' broken fan graphs. Let $F_n(t)$ denote the broken fan graph on $n$ vertices such that the spokes are uniformly distributed every $t$ vertices. A broken fan graph $F_8(2)$ is shown in Figure~\ref{fig:f82}. Suppose there are $r$ spokes in $F_n(t)$. Then $n= (r-1)t +2$ and $F_n(t)$ has $(r-1)t +r$ edges.
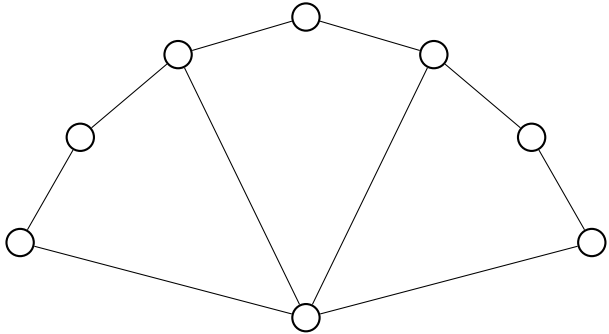
\begin{figure}[htb]
\centering
\begin{tikzpicture}[every node/.style={circle,thick,draw}] 
\node (1) at (4, 0) {};
\node (2) at (7.8, 1) {};
\node (3) at (7, 2.4) {};
\node (4) at (5.7, 3.5) {};
\node (5) at (4, 4) {};
\node (6) at (2.3, 3.5) {};
\node (7) at (1, 2.4) {};
\node (8) at (0.2, 1) {};
\begin{scope}[>={},every node/.style={fill=white,circle,inner sep=0pt,minimum size=12pt}]
\path [] (1) edge (2);
\path [] (1) edge (4);
\path [] (1) edge (6);
\path [] (1) edge (8);
\path [] (2) edge (3);
\path [] (3) edge (4);
\path [] (4) edge (5);
\path [] (5) edge (6);
\path [] (6) edge (7);
\path [] (7) edge (8);
\end{scope}
\end{tikzpicture}
\caption{A broken fan graph $F_8(2)$}
\label{fig:f82}
\end{figure}

We will show that $F_n(t)$ has a CBO. We first deal with the case of $t=2$, and then generalize the idea to any $t\ge 2$.
\begin{theorem}\label{brokenfan:fn2}
$F_n(2)$ has a cyclic base ordering.
\end{theorem}
\begin{proof}
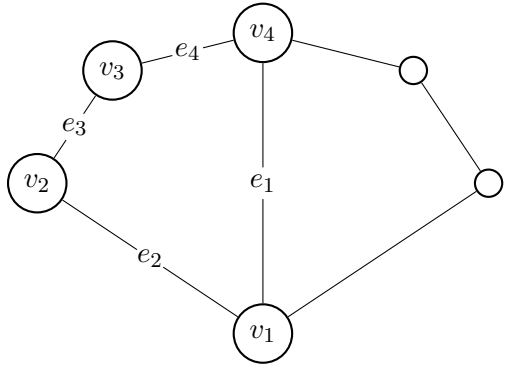
\begin{figure}[htb]
\centering
\begin{tikzpicture}[every node/.style={circle,thick,draw}] 
\node (1) at (4, 0) {$v_1$};
\node (2) at (7, 2) {};
\node (3) at (6, 3.5) {};
\node (4) at (4, 4) {$v_4$};
\node (5) at (2, 3.5) {$v_3$};
\node (6) at (1, 2) {$v_2$};
\begin{scope}[>={},every node/.style={fill=white,circle,inner sep=0pt,minimum size=12pt}]
\path [] (1) edge (2);
\path [] (1) edge node {$e_1$} (4);
\path [] (1) edge node {$e_2$} (6);
\path [] (2) edge (3);
\path [] (3) edge (4);
\path [] (4) edge node {$e_4$} (5);
\path [] (5) edge node {$e_3$} (6);
\end{scope}
\end{tikzpicture}
\caption{Induction for $F_n(2)$}
\label{induction:f62}
\end{figure}
Clearly $n=2r$, where $r$ is the number of spokes. We will prove this by induction on $r$. Our inductive hypothesis is for all smaller $r$, $F_n(2)$ where $n=2r$ has a cyclic base ordering. Our base case is when $r=2$, which is $C_4$ (a cycle on 4 vertices). Let $G$ be the $F_n(2)$ where $n=2r$. Then $G$ has $3r-2$ edges. Suppose the universal vertex is $v_1$ and other vertices are $v_2, v_3,\ldots, v_n$ on a path $P_{n-1}$.

As shown in Figure~\ref{induction:f62}, let $e_1 = v_4v_1$, $e_2 = v_1v_2$, $e_3=v_2v_3$, $e_4=v_3v_4$. Let $G'$ be $G - v_2 - v_3$, which is a $F_n(2)$ where $n=2(r-1)$. By our inductive hypothesis, $G'$ has a cyclic base ordering, denoted by $\mathcal{O}'$. Without loss of generality, let $\mathcal{O}'(e_1) = 1$. We define an edge ordering $\mathcal{O}$ of $G$ as the following:
\begin{align*}
\mathcal{O}(e) = \begin{cases}
1 & \text{ if } e = e_1 \\
2 & \text{ if } e = e_2\\
\mathcal{O}'(e)+1 & \text{ if } e \in G' \text{ and } 2 \leq \mathcal{O}'(e) \leq r-1 \\
r+1 & \text{ if } e = e_3\\
\mathcal{O}'(e)+2 & \text{ if } e \in G' \text{ and } r \leq \mathcal{O}'(e) \leq 2r-3 \\
2r & \text{ if } e = e_4\\
\mathcal{O}'(e)+3 & \text{ if } e \in G' \text{ and } 2r-2 \leq \mathcal{O}'(e) \leq 3r-5.
\end{cases}
\end{align*}
It can be seen that every progression that contains exactly two of the newly added edges in $\{e_2, e_3, e_4\}$ forms a spanning tree. The only progression that does not have exactly two of the newly added edges is $(e_2, \mathcal{O}^{-1}(3), \mathcal{O}^{-1}(4), \ldots, e_3, \mathcal{O}^{-1}(r+2), \mathcal{O}^{-1}(r+3), \ldots, e_4)$, which contains all three edges from $\{e_2, e_3, e_4\}$. If this progression contains a cycle, then it must contains all of $e_2, e_3, e_4$. However, replacing $e_2, e_3, e_4$ by $e_1$ will result a cycle, violating $\mathcal{O}'$ is a CBO of $G'$. Thus this progression must induce a spanning tree, and so $\mathcal{O}$ is a cyclic base ordering of $G$.
\end{proof}

\begin{theorem}\label{brokenfan:uniform}
$F_n(t)$ has a cyclic base ordering.
\end{theorem}
\begin{proof}
Clearly $n=(r-1)t+2$ and $F_n(t)$ has $(r-1)t+r$ edges. We use the same notations as Theorem~\ref{brokenfan:fn2}, and let $e_1 = v_1v_{t+2}$, $e_2 = v_1v_2$ and $e_i = v_{i-1}v_i$ for $i=3,4,\ldots, t+2$. 
Here is an example for $t=3$.
\begin{figure}[htb]
\centering
\begin{tikzpicture}[every node/.style={circle,thick,draw}] 
\node (1) at (4, 0) {$v_1$};
\node (2) at (7.8, 1) {};
\node (3) at (7, 2.4) {};
\node (4) at (5.7, 3.5) {};
\node (5) at (4, 4) {$v_5$};
\node (6) at (2.3, 3.5) {$v_4$};
\node (7) at (1, 2.4) {$v_3$};
\node (8) at (0.2, 1) {$v_2$};
\begin{scope}[>={},every node/.style={fill=white,circle,inner sep=0pt,minimum size=12pt}]
\path [] (1) edge (2);
\path [] (1) edge node {$e_1$} (5);
\path [] (1) edge node {$e_2$} (8);
\path [] (2) edge (3);
\path [] (3) edge (4);
\path [] (4) edge (5);
\path [] (5) edge node {$e_5$} (6);
\path [] (6) edge node {$e_4$} (7);
\path [] (7) edge node {$e_3$} (8);
\end{scope}
\end{tikzpicture}
\caption{Induction for $F_n(3)$}
\label{fig:fn3}
\end{figure}
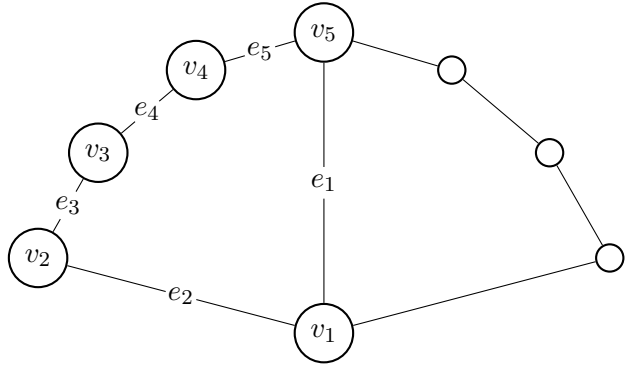

Define the following edge ordering of $G$ from a CBO $\mathcal{O'}$ of $G' = G-\{v_2,v_3,\cdots, v_{t+1}\}$, where $\mathcal{O}'(e_1) = 1$:
\begin{align*}
\mathcal{O}(e) = \begin{cases}
1 & \text{ if } e = e_1 \\
2 & \text{ if } e = e_2\\
\mathcal{O}'(e)+1 & \text{ if } e \in G' \text{ and } 2 \leq \mathcal{O}'(e) \leq r-1 \\
r+1 & \text{ if } e = e_3\\
\mathcal{O}'(e)+2 & \text{ if } e \in G' \text{ and } r \leq \mathcal{O}'(e) \leq 2r-3 \\
2r & \text{ if } e = e_4\\
\mathcal{O}'(e)+3 & \text{ if } e \in G' \text{ and } 2r-2 \leq \mathcal{O}'(e) \leq 3r-5 \\
\ldots \\
(r-1)t+2 & \text{ if } e = e_{t+2}\\
\mathcal{O}'(e)+r & \text{ if } e \in G' \text{ and } (r-2)t+2 \leq \mathcal{O}'(e) \leq (r-2)(t+1)+1.
\end{cases}
\end{align*}

If a progression contains exactly $t$ of the newly added edges in $\{e_2, e_3, \ldots, e_{t+2}\}$, it forms a spanning tree. The only progression that doesn't have exactly $t$ of the newly added edges is $(e_2, \mathcal{O}^{-1}(3), \mathcal{O}^{-1}(4), \ldots, e_{t+2})$. It is like to replaces $e_1$ with $\{e_2, e_3, \ldots, e_{t+2}\}$ in the progression, and so $\mathcal{O}$ is a cyclic base ordering of $G$.
\end{proof}

\section{Broken wheel graphs}
A {\bf broken wheel graph} on $n\ge 4$ vertices is a graph obtained from a wheel graph $W_n$ by removing some of the spokes. In general, broken wheel graphs may not be uniformly dense, and thus may not have a CBO. Here we give a direct proof without using uniformly dense condition.

\subsection{General broken wheel graphs}

\begin{lemma}\label{bw:lb}
A graph with a cyclic base ordering must have no vertex with degree less than $\frac{|E|}{|V|-1}$.
\end{lemma}

\begin{proof}
For each vertex in graph $G$, there must be at least one edge connected to it in order for the edges to form a spanning tree. Therefore, in a cyclic base ordering, each vertex's edges must have a cyclic spacing of at most $|V|-1$. Because there are a total of $|E|$ edges, the minimum number of edges a vertex must have to satisfy the required spacing is $\frac{|E|}{|V|-1}$.
\end{proof}

\begin{theorem}
If a broken wheel graph has a CBO, then the number of spokes must be at least $2$.
\end{theorem}
\begin{proof}
Let $G$ be a broken wheel graph with $n$ vertices and $r$ remaining edges. For this graph, $|V| = n$ and $|E| = n+r-1$. From Lemma \ref{bw:lb}, we know that the minimum degree of any vertex must be at least $\frac{|E|}{|V|-1}$ to have a cyclic base ordering. In this case, the minimum degree of any vertex must be at least $\frac{n+r-1}{n-1}$. We can determine the minimum amount of inside edges required with $r \geq \frac{n+r-1}{n-1}$, or $r \geq \frac{n-1}{n-2}$. Since $\frac{n-1}{n-2}$ asymptotically approaches $1$ but never reaches it, in order for a broken wheel graph to have a cyclic base ordering, $r > 1$.
\end{proof}

\begin{theorem}\label{bw:cyclelb}
If a broken wheel graph $G$ has a CBO, then the smallest cycle of $G$ must have at least $\frac{n-1}{r} + 1$ vertices.
\end{theorem}
\begin{proof}
When a cycle exists in a graph that has a cyclic base ordering, the edges in the cycle must not all be within $n-1$ edges of each other cyclically. Because the optimal way to space edges out cyclically is by evenly distributing them, we can determine if this optimal spacing of edges is within $n-1$ edges of each other. We let $s$ be the number of vertices in the smallest cycle in the graph. The spacing between each edge in a cycle of size $s$ is $\frac{n+r-1}{s}$, and since we can have at most $s-1$ vertices in the spanning tree at any time, the largest spanning tree a cycle with $s-1$ vertices can be in is $\frac{s-1}{s}(n+r-1)$. Therefore, $n-1 \leq \frac{s-1}{s}(n+r-1)$, which can be rearranged to $s \geq \frac{n-1}{r}+1$.
\end{proof}

For example, a broken wheel graph with $n=11,r=3,s=4$ (Figure~\ref{fig:bw11}) cannot have a cyclic base ordering due to Theorem \ref{bw:cyclelb}. 
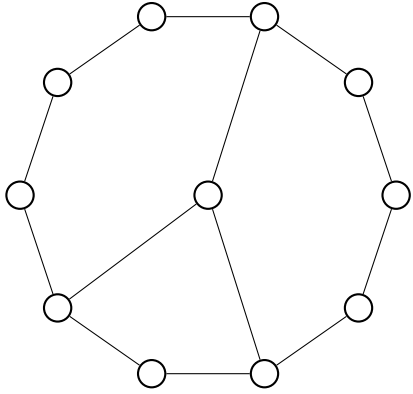
\begin{figure}[htb]
\centering{
\begin{tikzpicture}[every node/.style={circle,thick,draw}] 
\node (1) at (3.25, 0.125) {};
\node (2) at (1.75, 0.125) {};
\node (3) at (0.5, 1) {};
\node (4) at (0, 2.5) {};
\node (5) at (0.5, 4) {};
\node (6) at (1.75, 4.875) {};
\node (7) at (3.25, 4.875) {};
\node (8) at (4.5, 4) {};
\node (9) at (5, 2.5) {};
\node (10) at (4.5, 1) {};
\node (11) at (2.5, 2.5) {};
\begin{scope}[>={},every node/.style={fill=white,circle,inner sep=0pt,minimum size=12pt}]
\path [] (1) edge (2);
\path [] (2) edge (3);
\path [] (3) edge (4);
\path [] (4) edge (5);
\path [] (5) edge (6);
\path [] (6) edge (7);
\path [] (7) edge (8);
\path [] (8) edge (9);
\path [] (9) edge (10);
\path [] (10) edge (1);
\path [] (1) edge (11);
\path [] (3) edge (11);
\path [] (7) edge (11);
\end{scope}
\end{tikzpicture}
}
\caption{Broken Wheel with $n=11,r=3,s=4$}
\label{fig:bw11}
\end{figure}

However, we can show that some families of broken wheel graphs have CBOs in the next two subsections.

\subsection{Broken wheel graphs missing one spoke}
We consider a specific type of broken wheel graphs with exactly one spoke is removed. We will provide an explicit construction to show that such a graph has a CBO.

\begin{theorem}
Any broken wheel graph missing exactly one spoke has a CBO.
\end{theorem}
\begin{proof}
Suppose that the universal vertex is $v_0$ and other vertices are $v_1,v_2,\ldots, v_{n-1}$ in the order on the rim. Without loss of generality, we may assume that the spoke $v_0v_1$ is removed.

For an even $n$, i.e., $n=2k+2$ for some positive integer $k$, we define an ordering $\mathcal{O}_{2k+2} = (v_1 v_2, v_0 v_2, v_{k+2} v_{k+3}, v_0 v_3, v_{k+3} v_{k+4}, \ldots, v_0 v_{k+1}, v_{2k + 1} v_1, v_0 v_{k+2}, v_2 v_3, v_0 v_{k+3}, v_3 v_4,
\ldots, v_0 v_{2k + 1}, v_{k+1} v_{k+2})$. Here is an illustration for $n = 2k + 2 = 8$ in Figure~\ref{fig:bwg}.
\begin{figure}[htb]
\centering
\begin{tikzpicture}[every node/.style={circle,thick,draw}] 
\node (1) at (2.5, 0) {$v_1$};
\node (2) at (4.5, 1) {$v_2$};
\node (3) at (5, 3) {$v_3$};
\node (4) at (3.5, 5) {$v_4$};
\node (5) at (1.5, 5) {$v_5$};
\node (6) at (0, 3) {$v_6$};
\node (7) at (0.5, 1) {$v_7$};
\node (8) at (2.5, 2.7) {$v_0$};
\begin{scope}[>={},every node/.style={fill=white,circle,inner sep=0pt,minimum size=12pt}]
\path [] (2) edge node {2} (8);
\path [] (3) edge node {4} (8);
\path [] (4) edge node {6} (8);
\path [] (5) edge node {8} (8);
\path [] (6) edge node {10} (8);
\path [] (7) edge node {12} (8);

\path [] (1) edge node {1} (2);
\path [] (3) edge node {11} (4);
\path [] (2) edge node {9} (3);
\path [] (5) edge node {3} (6);
\path [] (4) edge node {13} (5);
\path [] (7) edge node {7} (1);
\path [] (6) edge node {5} (7);
\end{scope}
\end{tikzpicture}
\caption{CBO of a broken wheel graph on 8 vertices missing 1 spoke}
\label{fig:bwg}
\end{figure}
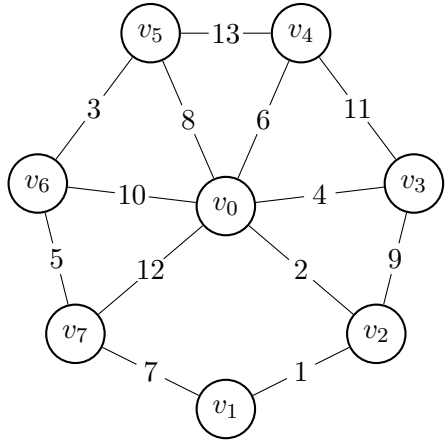

Suppose there are two sets of edges, $a$ and $b$. For $i=1,2,\ldots, n$, the edge $a_i$ joins $v_i$ to $v_{i+1}$, and $v_{n} = v_1$ for the sake of this definition. For $i=1,2,\ldots, n-2$, the edge $b_i$ joins $v_0$ and $v_{i+1}$.

Then, in general, the proposed CBO starts with $a_1, b_1$, and then the pattern that follows, or the first half of the CBO, is $a_{\tfrac{n}{2} + 1}, b_2, a_{\tfrac{n}{2} + 2}, b_3, \ldots, a_{n-2}, b_{\tfrac{n}{2} - 1}, a_{n-1}$. The second half immediately follows, and has the pattern $b_{\tfrac{n}{2}}, a_{2}, b_{\tfrac{n}{2} + 1}, a_{3}, \ldots, b_{\tfrac{n}{2}}, a_{\tfrac{n}{2}}$.

We will do casework on which edge is the first edge and prove that there is always a cycle in every cyclically consecutive $n-1$ edges. 

If the first edge is $a_{i}$ for some $i, 1 < i \leq \dfrac{n}{2}$, then the sequence of edges is 

$a_i, b_{\tfrac{n}{2} + i - 1}, a_{i+1}, b_{\tfrac{n}{2} + i}, \ldots, a_{\tfrac{n}{2} - 1}, b_{n-2}, a_{\tfrac{n}{2}} a_1, b_1$, and then $a_{\tfrac{n}{2} + 1}, b_2, a_{\tfrac{n}{2} + 2}, b_3, \ldots, a_{\tfrac{n}{2} + i - 2}, b_{i-1}$.

The first and second halves of the $a$ edges connects vertices from $v_i$ to  $v_{\tfrac{n}{2} + i - 1}$ (a connected component), as well as $v_1$ to $v_2$. 

The two halves of $b$ edges connects all the vertices from $v_2$ to $v_i$ to $v_0$ and all the vertices from $v_{\tfrac{n}{2} + i}$ to $v_{n - 1}$ (making all these vertices in a connected component).

Thus, the two connected components from the $a$ and $b$ edges are joined at $v_i$, connecting all the edges from $v_2$ to $v_{n-1}$ in a single component, and finally, $v_1$ is connected to $v_2$, making a spanning tree.

To get the corresponding sequences starting with a $b$ edge, we can shift it so that $b_{\tfrac{n}{2} + i - 1}$ is the first edge. Then, the extra edge at the end would be $a_{\tfrac{n}{2} + i - 1}$. The $b$ edges still connect all the vertices from $v_2$ to $v_i$ to $v_0$ and all the vertices from $v_{\tfrac{n}{2} + i}$ to $v_{n - 1}$. However, the $a$ edges now no longer connect $v_i$ among them but instead connect $v_{\tfrac{n}{2} + i}$. $v_{\tfrac{n}{2} + i}$ then becomes the new joining point for the cycles, so a spanning tree is still formed.

Similarly, if the first edge is $a_{\tfrac{n}{2} + i}$, where $1 \leq i < \dfrac{n}{2}$, the edge sequence would be 

$a_{\tfrac{n}{2} + i}, b_{i + 1}, a_{\tfrac{n}{2} + i + 1}, b_{i + 2}, \ldots, a_{n - 1}, b_{\tfrac{n}{2}}$, and then $a_2, b_{\tfrac{n}{2} + 1}, a_3, b_{\tfrac{n}{2} + 2}, \ldots, a_i,  b_{\tfrac{n}{2} + i - 1}, a_{i+1}$.

Then, from the $a$ set of edges, all the edges from $v_{\tfrac{n}{2} + i}$ to $v_n$ (or just $v_1$) would be in a connected component, and the vertices from $v_2$ to $v_{i + 2}$ would be in another connected component. 

From the $b$ set of edges, all vertices from $v_{i+2}$ to $v_{\tfrac{n}{2} + i}$ would be in a connected component (along with $v_0$). The three components together contain all the vertices, and they are joined together by the common vertices $v_{i+2}$ and $v_{\tfrac{n}{2} + i}$, so a spanning tree is formed.

To get the corresponding sequences starting with a $b$ edge, we can shift it so that $b_{i+1}$ is the first edge. Then, the new edge at the end would be $b_{\tfrac{n}{2} + i}$. Since $a_{\tfrac{n}{2} + i}$ is gone, $v_{\tfrac{n}{2} + i}$ would no longer be connected by the $a$ edges, so it is no longer a joining point for the separate components. However, the addition of $b_{\tfrac{n}{2} + i}$ connects $v_{\tfrac{n}{2} + i + 1}$ to the $b$ edge connected component, so $v_{\tfrac{n}{2} + i + 1}$ becomes the new joining point and a spanning tree is still formed.

If the first edge is $a_1$, then the sequence is $a_1, b_1, a{\tfrac{n}{2} + 1}, b_2, a{\tfrac{n}{2} + 2}, b_3, \ldots, b_{\tfrac{n}{2} - 1}, a_{n-1}$. The $a$ edges connect the vertices from $v_{\tfrac{n}{2} + 1}$ to $v_n$ (or just $v_1$) and $v_1$ to $v_2$, and the $b$ edges connect all the vertices from $v_2$ to $v_{\tfrac{n}{2}}$ to $v_0$, so together a spanning tree is still formed.

\bigskip
For an odd $n$, i.e., $n=2k+1$ for some positive integer $k\ge 2$, we define an ordering $\mathcal{O}_{2k+1} = (a_{\tfrac{n+1}{2}}, b_1, a_{\tfrac{n+3}{2}}, b_2, \ldots, a_{n-1}, b_{\tfrac{n-1}{2}}, b_{\tfrac{n+1}{2}}, a_2, b_{\tfrac{n+3}{2}}, a_3,\ldots, b_{n-2}, a_{\tfrac{n-1}{2}}, a_1)$, where the definitions of $a_i$ and $b_i$ is the same as above.

Similar to the even $n$ case, the CBO has 2 "halves", the first half ending at $b_{\tfrac{n-1}{2}}$, and the second ending at $a_{\tfrac{n-1}{2}}$ with $a_1$ being the last extra edge.

The first case is when the cyclic edges starts from $a_{\tfrac{n+1}{2} + k}$, where $0 \leq k < \tfrac{n-1}{2}$. In this case, the edges would be  $(a_{\tfrac{n+1}{2} + k}, b_{k+1}, a_{\tfrac{n+1}{2} + k + 1}, b_{k+2}, \ldots a_{n-1}, b_{\tfrac{n-1}{2}}, b_{\tfrac{n+1}{2}}, a_2, b_{\tfrac{n+3}{2}}, a_3, \ldots, b_{\tfrac{n-1}{2} + k}, a_{k+1})$

The $a$ edges join $v_2, v_3, \ldots, v_{k+2}$ as well as $v_{\tfrac{n+1}{2} + k}, v_{\tfrac{n+1}{2} + k + 1}, \ldots, v_{n-1}, v_n$
($v_n$ is just $v_1$ here) in 2 connected components, while the $b$ edges join $v_{k+2}, v_{k+3}, \ldots v_{\tfrac{n+1}{2}}$ and $v_{\tfrac{n+1}{2} + 1}, v_{\tfrac{n+1}{2} + 2}, \ldots v_{\tfrac{n+1}{2} + k}$ (all in a single component connected to $v_0$). $v_{\tfrac{n+1}{2} + k}$ and $v_{k+2}$ act as the 2 joining points for the 3 connected components.

The corresponding sequences of cyclic edges starting from a $b$ edge would exclude $a_{\tfrac{n+1}{2} + k}$ but include $b_{\tfrac{n+1}{2} + k}$. All this would do would be to change the 2 $v_{\tfrac{n+1}{2} + k}$ vertices at the end of the 2nd two connected components to $v_{\tfrac{n+1}{2} + k + 1}$, making $v_{\tfrac{n+1}{2} + k + 1}$ the new joining point for the components instead of $v_{\tfrac{n+1}{2} + k}$. This means that a spanning tree is still formed.

Another possible starting point is at $b_{\tfrac{n+1}{2} + k}$, where $0 \leq k < \tfrac{n-3}{2}$. In this case, the cyclic edges $b_{\tfrac{n+1}{2} +k}, a_{k+2}, b_{\tfrac{n+1}{2} + k + 1}, a_{k+3}, \ldots, b_{n-2}, a_{\tfrac{n-1}{2}}, a_1, a_{\tfrac{n+1}{2}}, b_1, a_{\tfrac{n+1}{2} + 1}, b_2, \ldots, a_{\tfrac{n-1}{2} + k}, b_k, a_{\tfrac{n+1}{2} + k}$.

The $a$ edges join $v_{k+2}, v_{k+3}, \ldots, v_{\tfrac{n+1}{2} + k + 1}$ as well as $v_1$ to $v_2$, and the $b$ edges join $v_{\tfrac{n+1}{2} + k + 1}$, $v_{\tfrac{n+1}{2} + k + 2}, \ldots, v_{n-1}$ as well as $v_2, v_3, \ldots, v_k, v_{k+1}$ (all joined to $v_0$). This covers all the vertices in the graph, so a spanning tree is formed once again. 

The corresponding cyclically consecutive vertices which start from $a$ instead are shifted over by 1 - they start from $a_{k+2}$ instead. In this case, $v_{\tfrac{n+1}{2} + k + 1}$ is no longer the "joiner" since the $b$ edges do not contain it, but $b_{k+1}$ is the extra edge added, so $v_{k+2}$ is now joined by the $b$ edges, making $v_{k+2}$ the new "joiner", so a CBO is once again formed.
\end{proof}

\subsection{Broken wheel graphs $W(n, r)$}
Suppose that the universal vertex is $v_0$ and vertices on the rim are $v_1,v_2,\ldots, v_{n-1}$. For $n-1\ge r\ge 2$, let $W(n, r)$ be a broken wheel graph on $n$ vertices and $r$ spokes such that the $r$ spokes are $v_0v_k$, where $k=\lfloor(n-1)i/r\rfloor +1$ for all integers $0\le i < r$. In other words, the spokes in $W(n, r)$ are uniformly distributed as much as possible. It has been proved in \cite{CaGH88} that $W(n, r)$ is uniformly dense for $n-1\ge r\ge 2$.

The broken wheel graphs missing one spoke we discussed in the last subsection actually is $W(n,n-2)$. We can generalize it and prove the following theorem.
\begin{theorem}
For $n-1 \geq r \geq 2$ and $n=rt+2$, $W(n,r)$ has a cyclic base ordering.
\end{theorem}
\begin{proof}
Let $G$ be the graph $W(n, r)$ where $n-1 \geq r \geq 2$ and $n=rt+2$. It can be seen that there are $r-1$ pairs of spokes with a gap of $t$ edges and $t-1$ vertices in between them. Additionally, there is one pair of spokes with a gap of $t+1$ edges and $t$ vertices. We denote the $t$ vertices as $S = (v_1, v_2, \ldots, v_t)$. Let $G'$ be $G - S$. $G'$ is the graph $F_{n-t}(t)$, and because $n-t=(r-1)t+2$, by Theorem \ref{brokenfan:uniform}, $G'$ has a cyclic base ordering which we call $\mathcal{O}'$. Let $x$ and $y$ be the two vertices with degree 2 incident to the universal vertex in $G'$. In other words, they are the outside edges of the rim in $G'$. Let $H$ be the unique direct path from $x$ to $y$ in the spanning tree of $G'$ formed by edges $(\mathcal{O}'^{-1}(1), \mathcal{O}'^{-1}(2), \ldots, \mathcal{O}'^{-1}(n-2))$. Let $p$ be the edge with the smallest $\mathcal{O}'(p)$ value in $H$. In other words, $p=\mathcal{O}'^{-1}(\min_{x \in H} \mathcal{O}'(x))$. It is easy to see that when edge $p$ is removed, vertices $x$ and $y$ will be in different components. Without loss of generality, let $\mathcal{O}'(p)=1$. Let the cyclic base ordering of $G$ be the following:
\begin{align}
\mathcal{O}(e) = \begin{cases}
1 & \text{ if } e = p \\
2 & \text{ if } e = v_1x\\
\mathcal{O}'(e)+1 & \text{ if } e \in G' \text{ and } 2 \leq \mathcal{O}'(e) \leq r \\
r+2 & \text{ if } e = v_1v_2\\
\mathcal{O}'(e)+2 & \text{ if } e \in G' \text{ and } r+1 \leq \mathcal{O}'(e) \leq 2r-1 \\
\ldots \\
rt+2 & \text{ if } e = v_ty\\
\mathcal{O}'(e)+r+1 & \text{ if } e \in G' \text{ and } (r-1)t + 2 \leq \mathcal{O}'(e) \leq (r-1)(t+1)+1
\end{cases}
\end{align}

\begin{figure}[htb]
\centering
\begin{tikzpicture}[every node/.style={circle,thick,draw}] 
\node (1) at (2.5, 0) {};
\node (2) at (4.5, 1) {};
\node (3) at (5, 3) {$y$};
\node (4) at (3.5, 5) {$v_2$};
\node (5) at (1.5, 5) {$v_1$};
\node (6) at (0, 3) {$x$};
\node (7) at (0.5, 1) {};
\node (8) at (2.5, 2.7) {};
\begin{scope}[>={},every node/.style={fill=white,circle,inner sep=0pt,minimum size=12pt}]
\path [] (3) edge (8);
\path [] (6) edge (8);
\path [] (1) edge (8);

\path [] (1) edge (2);
\path [] (3) edge (4);
\path [] (2) edge (3);
\path [] (5) edge (6);
\path [] (4) edge (5);
\path [] (7) edge (1);
\path [] (6) edge (7);
\end{scope}
\end{tikzpicture}
\caption{$W(8, 3)$}
\end{figure}
\end{proof}

\section{Prism graphs}
A {\bf prism graph} $Y_n$ is a graph that has a $n$-gonal prism as its skeleton. Clearly $Y_n$ has $2n$ vertices and $3n$ edges. Equivalently, a prism graph $Y_n$ can be constructed as the Cartesian product of the cycle $C_n$ and a single edge $K_2$. Some cyclic base orderings of $Y_3$ and $Y_5$ are shown in Figure~\ref{fig:y3}.
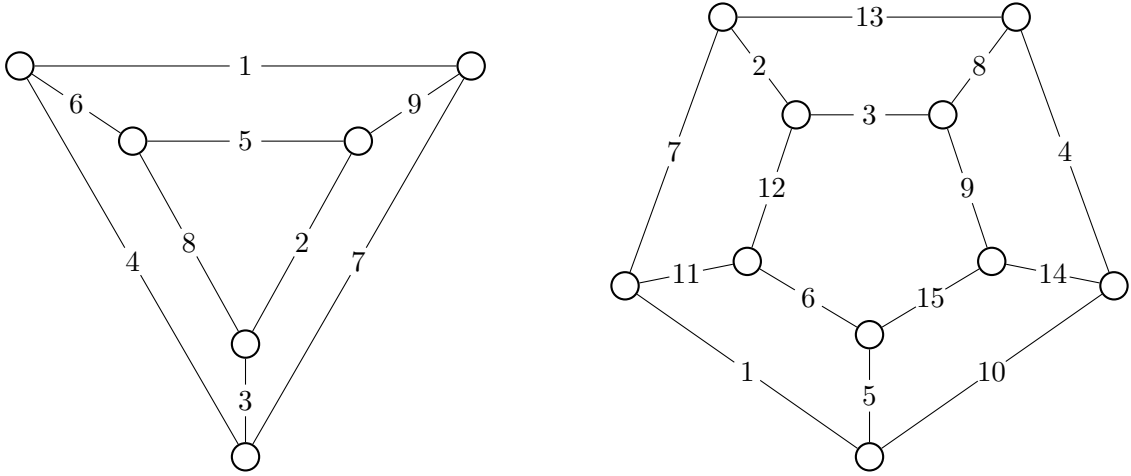
\begin{figure}[htb]
\begin{subfigure}
\centering
\begin{tikzpicture}[every node/.style={circle,thick,draw}] 
\node (1) at (5, 0) {};
\node (2) at (2, 5.2) {};
\node (3) at (8, 5.2) {};
\node (4) at (5, 1.5) {};
\node (5) at (3.5, 4.2) {};
\node (6) at (6.5, 4.2) {};
\begin{scope}[>={},every node/.style={fill=white,circle,inner sep=0pt,minimum size=12pt}]
\path [] (1) edge node {4} (2);
\path [] (2) edge node {1} (3);
\path [] (3) edge node {7} (1);
\path [] (4) edge node {8} (5);
\path [] (5) edge node {5} (6);
\path [] (6) edge node {2} (4);
\path [] (1) edge node {3} (4);
\path [] (2) edge node {6} (5);
\path [] (3) edge node {9} (6);
\end{scope}
\end{tikzpicture}
\end{subfigure}
\hspace{0.5in}
\begin{subfigure}
\centering
\begin{tikzpicture}[every node/.style={circle,thick,draw}, scale=0.65] 
\node (1) at (5, 0){};
\node (2) at (0, 3.5) {};
\node (3) at (2, 9) {};
\node (4) at (8, 9) {};
\node (5) at (10, 3.5) {};
\node (6) at (5, 2.5) {};
\node (7) at (2.5, 4) {};
\node (8) at (3.5, 7) {};
\node (9) at (6.5, 7) {};
\node (10) at (7.5, 4) {};
\begin{scope}[>={},every node/.style={fill=white,circle,inner sep=0pt,minimum size=12pt}]
\path [] (1) edge node {1} (2);
\path [] (10) edge node {15} (6);
\path [] (4) edge node {8} (9);
\path [] (2) edge node {7} (3);
\path [] (10) edge node {9} (9);
\path [] (5) edge node {14} (10);
\path [] (3) edge node {13} (4);
\path [] (8) edge node {3} (9);
\path [] (1) edge node {5} (6);
\path [] (4) edge node {4} (5);
\path [] (6) edge node {6} (7);
\path [] (2) edge node {11} (7);
\path [] (1) edge node {10} (5);
\path [] (7) edge node {12} (8);
\path [] (3) edge node {2} (8);
\end{scope}
\end{tikzpicture}

\end{subfigure}
\caption{$Y_3$ and $Y_5$}
\label{fig:y3}
\end{figure}

\begin{theorem}
If $n=3m+2$ for some positive integer $m$, then $Y_n$ has a CBO.
\end{theorem}
\begin{proof}
Any prism graph $Y_n$ has an outer cycle $C_n$, an inner cycle $C_n$, and edges between the outer cycle and the inner cycle. Suppose that the vertices of the outer cycles are $v_1, v_2,\ldots, v_n$, while the vertices of the inner cycles are $u_1,u_2,\ldots, u_n$.

Now, define 3 sets of edges, $a$, $b$, $c$. For every integer $i$ such that $1 \leq i \leq n$, define $a_i$ to be the edge $v_i v_{i+1}$, $b_i$ as the edge $v_i u_i$, and $c_i$ as the edge $u_i u_{i+1}$. When the subscript $k > n$ or $k \leq 0$, $v_k, a_k, b_k, c_k$ are the same as $v_r, a_r, b_r, c_r$ respectively, where $r$ is the (unique) integer such that $1 \leq r \leq n$ and $r \equiv k \pmod{n}$.

Define an edge ordering $\mathcal{O}$ of $Y_n$ by
$$\mathcal{O} = (a_1, b_{3}, c_{3}, a_{4}, b_{6}, c_{6}, a_{7}, b_{9}, c_{9},\ldots ,a_{3n-2}, b_{3n}, c_{3n}).$$
This ordering can be considered as a list of $n$ ordered triples $(a_{3i+1}, b_{3i+3}, c_{3i+3})$ for $i=0,1,\ldots, n-1$. Here is an example for $Y_8$ in Figure~\ref{fig:y8}.

\begin{figure}[htb]
\begin{center}
\begin{tikzpicture}[every node/.style={circle,thick,draw}, scale=0.8] 
\node (1) at (5, 0) {$v_1$};
\node (2) at (1.5, 1.5) {$v_2$};
\node (3) at (0, 5) {$v_3$};
\node (4) at (1.5, 8.5) {$v_4$};
\node (5) at (5, 10) {$v_5$};
\node (6) at (8.5, 8.5) {$v_6$};
\node (7) at (10, 5) {$v_7$};
\node (8) at (8.5, 1.5) {$v_8$};
\node (9) at (5, 2) {$u_1$};
\node (10) at (2.8, 2.8) {$u_2$};
\node (11) at (2, 5) {$u_3$};
\node (12) at (2.8, 7.2) {$u_4$};
\node (13) at (5, 8) {$u_5$};
\node (14) at (7.2, 7.2) {$u_6$};
\node (15) at (8, 5) {$u_7$};
\node (16) at (7.2, 2.8) {$u_8$};
\begin{scope}[>={},every node/.style={fill=white,circle,inner sep=0pt,minimum size=12pt}]
\path [] (1) edge node {1} (2);
\path [] (2) edge node {10} (3);
\path [] (3) edge node {19} (4);
\path [] (4) edge node {4} (5);
\path [] (5) edge node {13} (6);
\path [] (6) edge node {22} (7);
\path [] (7) edge node {7} (8);
\path [] (8) edge node {16} (1);
\path [] (1) edge node {8} (9);
\path [] (2) edge node {17} (10);
\path [] (3) edge node {2} (11);
\path [] (4) edge node {11} (12);
\path [] (5) edge node {20} (13);
\path [] (6) edge node {5} (14);
\path [] (7) edge node {14} (15);
\path [] (8) edge node {23} (16);
\path [] (9) edge node {9} (10);
\path [] (10) edge node {18} (11);
\path [] (11) edge node {3} (12);
\path [] (12) edge node {12} (13);
\path [] (13) edge node {21} (14);
\path [] (14) edge node {6} (15);
\path [] (15) edge node {15} (16);
\path [] (16) edge node {24} (9);
\end{scope}
\end{tikzpicture}
\end{center}
\caption{CBO of $Y_8$}
\label{fig:y8}
\end{figure}
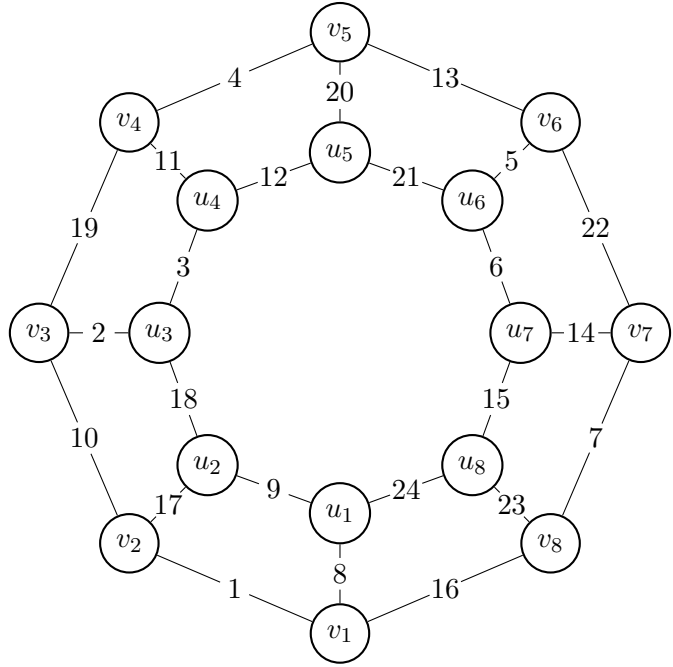

We will show that $\mathcal{O}$ is a CBO of $Y_n$. By symmetry, any progression starting with some $a_{3i+1}$ will always have the same structure. Similarly for any progression starting with some $b_{3i+3}$ or $c_{3i+3})$. Thus it suffices to verify the first three progressions. Notice that $n=3m+2$ and $Y_n$ has $2n$ vertices. Each progression has $2n-1 = 6m+3$ edges.

The first progression is $(a_1, b_{3}, c_{3}, a_{4}, b_{6}, c_{6}, a_{7}, b_{9}, c_{9},\ldots ,a_{6m+1}, b_{6m+3}, c_{6m+3})$. All the $a$-edges in this progression are $a_1, a_4,\ldots, a_{3m-2}, a_{3m+1}, a_{3m+4},\ldots, a_{6m+1}$. If the subscript $k$ is greater than $3m+2$, then the subscript is actually $k-(3m+2)$. Thus the $a$-edges in this progression actually are $a_1, a_4,\ldots, a_{3m-2}, a_{3m+1}, a_2,a_5,\ldots, a_{3m-1}$. In other words, the $a$-edges induce $m+1$ short paths $v_1v_2v_3$, $v_4v_5v_6$, $\ldots$, $v_{3m-2}v_{3m-1}v_{3m}$ and $v_{3m+1}v_{3m+2}$.

All the $c$-edges in this progression are $c_3, c_6,\ldots, c_{3m}, c_{3m+3},\ldots, c_{6m+3}$. If the subscript $k$ is greater than $3m+2$, then the subscript is actually $k-(3m+2)$. Thus the $c$-edges in this progression actually are $c_3, c_6,\ldots, c_{3m}, c_1, c_4, c_7,\ldots, c_{3m+1}$. In other words, the $c$-edges also induce $m+1$ short paths, they are $u_1u_2$, $u_3u_4u_5$, $u_6u_7u_8$, $\ldots$, $u_{3m}u_{3m+1}u_{3m+2}$.

All the $b$-edges in this progression are $b_3, b_6,\ldots, b_{3m}, b_{3m+3},\ldots, b_{6m+3}$. If the subscript $k$ is greater than $3m+2$, then the subscript is actually $k-(3m+2)$. Thus the $b$-edges in this progression actually are $b_3, b_6,\ldots, b_{3m}, b_1, b_4, b_7,\ldots, b_{3m+1}$. It is not hard to see that each $b_{3i}$ connects the path $v_{3i-2}v_{3i-1}v_{3i}$ and the path $u_{3i}u_{3i+1}u_{3i+2}$ for $i=1,2,\ldots,m$. Each $b_{3i+1}$ connects the path $v_{3i+1}v_{3i+2}v_{3i+3}$ and the path $u_{3i}u_{3i+1}u_{3i+2}$ for $i=1,2,\ldots,m-1$, and $b_1$ connects the path $v_1v_2v_3$ and the edge $u_1u_2$, while $b_{3m+1}$ connects the path $v_{3m+1}v_{3m+2}$ and the path $u_{3m}u_{3m+1}u_{3m+2}$. All vertices are connected by this progression and the progression contains exactly $2n-1$ edges, thus it induces a spanning tree.

The second progression is $(b_{3}, c_{3}, a_{4}, b_{6}, c_{6}, a_{7}, b_{9}, c_{9},\ldots ,a_{6m+1}, b_{6m+3}, c_{6m+3}, a_{6m+4})$. Notice that $a_{6m+4}$ is actually $a_{3m+2}$.
It is like to remove $a_1$ from the first progression and add $a_{3m+2}$. Removing $a_1$ will disconnect $v_1$ from other $v_k$'s but adding $a_{3m+2}$ will connect $v_1$ back to $v_{3m+2}$. Thus this progression still induces a spanning tree.

The third progression is $(c_{3}, a_{4}, b_{6}, c_{6}, a_{7}, b_{9}, c_{9},\ldots ,a_{6m+1}, b_{6m+3}, c_{6m+3}, a_{6m+4}, b_{6m+6})$. Notice that $b_{6m+6}$ is actually $b_2$. Similarly, it is like to remove $b_3$ from the second progression and add $b_{2}$. It is not hard to verify this progression still induces a spanning tree.

Therefore $\mathcal{O}$ is a CBO of $Y_n$.
\end{proof}

Here is an example for $Y_{11}$.
\begin{figure}[htb]
\begin{center}
\begin{tikzpicture}[every node/.style={circle,thick,draw}, scale=0.9] 
\node (1) at (5, 0) {$v_1$};
\node (2) at (2.2, 1.1) {$v_2$};
\node (3) at (0.2, 3.2) {$v_3$};
\node (4) at (-0.1, 6) {$v_4$};
\node (5) at (1.3, 8.3) {$v_5$};
\node (6) at (3.7, 10) {$v_6$};
\node (7) at (6.3, 10) {$v_7$};
\node (8) at (8.7, 8.3) {$v_8$};
\node (9) at (10.1, 6) {$v_9$};
\node (10) at (9.8, 3.2) {$v_{10}$};
\node (11) at (7.8, 1.1) {$v_{11}$};
\node (12) at (5, 2) {$u_1$};
\node (13) at (3.3, 2.7) {$u_2$};
\node (14) at (2, 4) {$u_3$};
\node (15) at (1.7, 5.8) {$u_4$};
\node (16) at (2.7, 7.2) {$u_5$};
\node (17) at (4.1, 8.2) {$u_6$};
\node (18) at (5.9, 8.2) {$u_7$};
\node (19) at (7.3, 7.2) {$u_8$};
\node (20) at (8.3, 5.8) {$u_9$};
\node (21) at (8, 4) {$u_{10}$};
\node (22) at (6.7, 2.7) {$u_{11}$};
\begin{scope}[>={},every node/.style={fill=white,circle,inner sep=0pt,minimum size=12pt}]
\path [] (1) edge node {1} (2);
\path [] (2) edge node {13} (3);
\path [] (3) edge node {25} (4);
\path [] (4) edge node {4} (5);
\path [] (5) edge node {16} (6);
\path [] (6) edge node {28} (7);
\path [] (7) edge node {7} (8);
\path [] (8) edge node {19} (9);
\path [] (9) edge node {31} (10);
\path [] (10) edge node {10} (11);
\path [] (11) edge node {22} (1);
\path [] (1) edge node {11} (12);
\path [] (2) edge node {23} (13);
\path [] (3) edge node {2} (14);
\path [] (4) edge node {14} (15);
\path [] (5) edge node {26} (16);
\path [] (6) edge node {5} (17);
\path [] (7) edge node {17} (18);
\path [] (8) edge node {29} (19);
\path [] (9) edge node {8} (20);
\path [] (10) edge node {20} (21);
\path [] (11) edge node {32} (22);
\path [] (12) edge node {12} (13);
\path [] (13) edge node {24} (14);
\path [] (14) edge node {3} (15);
\path [] (15) edge node {15} (16);
\path [] (16) edge node {27} (17);
\path [] (17) edge node {6} (18);
\path [] (18) edge node {18} (19);
\path [] (19) edge node {30} (20);
\path [] (20) edge node {9} (21);
\path [] (21) edge node {21} (22);
\path [] (22) edge node {33} (12);
\end{scope}
\end{tikzpicture}
\end{center}
\caption{CBO of $Y_{11}$}
\end{figure}
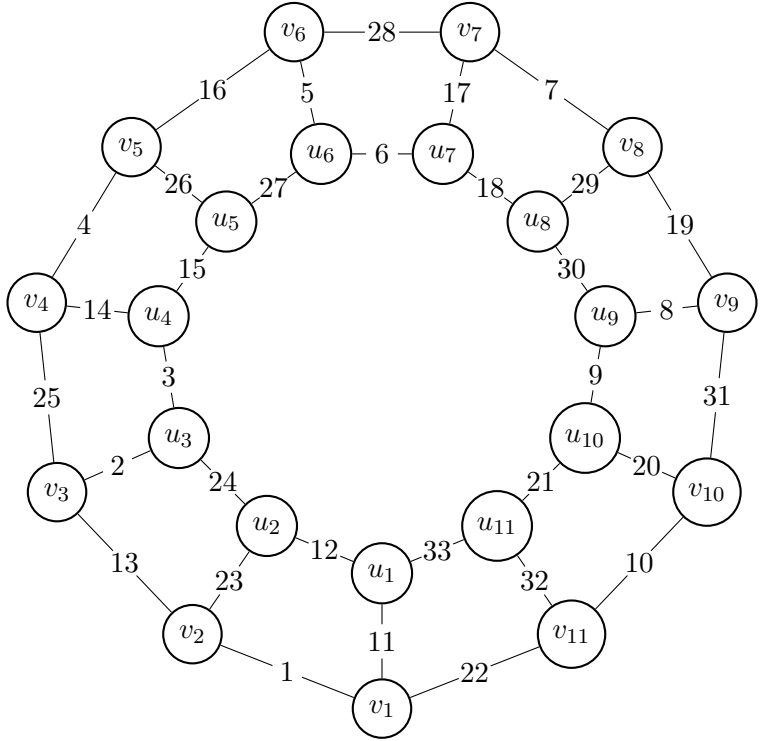

\section{Maximal $2$-degenerate graphs}
A graph is {\bf $k$-degenerate} if every subgraph has a vertex of degree at most $k$. A $k$-degenerate graph $G$ is {\bf maximal} if $G$ is no long $k$-degenerate after adding any additional edge. In particular, a $k$-tree is a maximal $k$-degenerate graph.

\begin{theorem}
Every maximal 2-degenerate graph has a cyclic base ordering.
\end{theorem}
\begin{proof}
A maximal 2-degenerate graph can be constructed from $K_2$ by adding a vertex and 2 edges at each step. 
We will use induction on the number of vertices, $n$, with the base case being $n=3$, that is $K_3$. Our inductive hypothesis is that all maximal 2-degenerate graphs with less than $n$ vertices have a cyclic base ordering. Let $G$ be a maximal 2-degenerate graph on $n$ vertices, and $v$ be the newly added vertex with degree 2. Have the two vertices adjacent to $v$ be called $x$ and $y$. It is clear that the graph $G-v$ is also a maximal 2-degenerate graph, so according to our inductive hypothesis, $G-v$ has a cyclic base ordering, denoted by $\mathcal{O}'$. Let $H$ be the unique path from $x$ to $y$ in the spanning tree of $G-v$ formed by edges $(\mathcal{O}'^{-1}(1), \mathcal{O}'^{-1}(2), \ldots, \mathcal{O}'^{-1}(n-2))$. Let $p$ be the edge with the smallest $\mathcal{O}'(p)$ value in $H$. In other words, $p=\mathcal{O}'^{-1}(\min_{x \in H} \mathcal{O}'(x))$. It is easy to see that when the edge $p$ is removed, vertices $x$ and $y$ will be in different components. Without loss of generality, let $\mathcal{O}'(p)=1$. We define an edge ordering $\mathcal{O}$ of $G$ as follows:
\begin{align*}
\mathcal{O}(e) = \begin{cases}
1 & \text{ if } e = p \\
2 & \text{ if } e = vx\\
\mathcal{O}'(e)+1 & \text{ if } e \in G-v \text{ and } 2 \leq \mathcal{O}'(e) \leq n-2 \\
n & \text{ if } e = vy\\
\mathcal{O}'(e)+2 & \text{ if } e \in G-v \text{ and } n-1 \leq \mathcal{O}'(e) \leq 2n-5.
\end{cases}
\end{align*}
Clearly, every progression of $\mathcal{O}$ that contains exactly one of $vx$ or $vy$ induces a spanning tree of $G$. The only progression that doesn't contain exactly one of $vx$ or $vy$ is the progression $(vx, \mathcal{O}^{-1}(3), \mathcal{O}^{-1}(4), \ldots, \mathcal{O}^{-1}(n-1), vy)$. However, the two edges $vx$ and $vy$ take the place of edge $p$ in connecting $x$ and $y$, with the graph formed by the $n-3$ edges between $vx$ and $vy$ having vertices $x$ and $y$ in different components. Thus no cycles are formed, and the progression induces a spanning tree of $G$. Therefore, $\mathcal{O}$ is a cyclic base ordering of $G$.
\end{proof}

\section{A polynomial time algorithm}

An edge ordering of a graph can be determined to be a cyclic base ordering in $\mathcal{O}(VE)$ time, where $V$ is the number of vertices and $E$ is the number of edges. This can be done by iterating through each starting edge within the edge ordering and determining if the following $V-1$ consecutive edges form a spanning tree. This is equivalent to finding if the $V-1$ consecutive edges are comprised of a single component. The number of components can be found by creating an adjacency list graph from the $V-1$ edges and running a depth-first search on the resulting graph. If there remain any unvisited vertices, then the edge ordering is not a cyclic base ordering. Each iteration takes $\mathcal{O}(V)$ time, making the total time complexity $\mathcal{O}(VE)$.

\begin{algorithm}[H]
\SetAlgoLined
    \SetKwFunction{FisCBO}{isCBO}
    \SetKwFunction{FDFS}{DFS}
    \SetKwProg{Fn}{Function}{:}{}
    \Fn{\FDFS{$v$, $H$, $visited$}}{
        \If{$visited[v]$}{
            \KwRet\;
        }
        $visited[v]=true$\;
        \For{$v'$ in $H[v]$}{
            DFS($v'$, $H$, $visited$)\;
        }
        \KwRet\;
    }
    \SetKwProg{Fn}{Function}{:}{}
    \Fn{\FisCBO{$G$}}{
        $e \gets$ edge ordering of graph $G$\;
        \For{$i \gets 0$ \KwTo $|e|$}{
            $H \gets$ adjacency list representation of $(e_i, e_{i+1 \mod |e|}, \ldots, e_{i+V-1 \mod |e|})$\;
            \For{$j \gets 0$ \KwTo $V$}{
                $visited[j] = false$\;
            }
            DFS($0$, $H$, $visited$)\;
            \For{$j \gets 0$ \KwTo $V$}{
                \If{not $visited[j]$}{
                    \KwRet $false$\;
                }
            }
        }
        \KwRet $true$\;
    }
 \caption{Cyclic Base Ordering in $\mathcal{O}(VE)$}
\end{algorithm}

This can be further optimized to amortized $\mathcal{O}(E \log V)$ time with a link-cut tree. A link-cut tree is a data structure representing a forest of trees that supports the following operations in amortized $\mathcal{O}(\log V)$ time: adding a new vertex to the forest, adding an edge from vertex $u$ to $v$ with $u$ as the parent, disconnecting a vertex from its parent, and querying the root of the tree a vertex is in. This data structure can be used to determine if a edge ordering is a cyclic base ordering. First, add all $V$ vertices to the data structure, taking amortized $\mathcal{O}(V \log V)$ time. Then, add the first $V-1$ edges to the link-cut tree, with the lower-indexed vertex as the parent, again taking amortized $\mathcal{O}(V \log V)$ time. Next, iterate through each edge ($e_i$). For each iteration, disconnect the parent of the higher-indexed vertex of $e_i$, and query the trees of the two vertices of $e_{i+V-1}$. If they are in the same tree, then the edge ordering is not a cyclic base ordering. If they are not in the same tree, connect the two vertices of $e_{i+V-1}$ and continue iterating. Since each iteration takes amortized $\mathcal{O}(\log V)$ time, the entire algorithm takes amortized $\mathcal{O}(E \log V)$ time.

\end{document}